\documentclass[12pt,a4paper]{amsart}
\usepackage{amsfonts}
\usepackage{amsthm}
\usepackage{amsmath}
\usepackage{amscd}
\usepackage{dsfont}
\usepackage[latin2]{inputenc}
\usepackage{t1enc}
\usepackage[mathscr]{eucal}
\usepackage{indentfirst}
\usepackage{graphicx}
\usepackage{graphics}
\usepackage{pict2e}
\usepackage{epic}
\numberwithin{equation}{section}
\usepackage[margin=2.9cm]{geometry}
\usepackage{stmaryrd}
\usepackage{url}

\newcommand{\p}{\mathbf{p}}
\newcommand{\q}{\mathbf{q}}

\newcommand{\bal}{\boldsymbol \alpha}

\theoremstyle{plain}
\newtheorem{theorem}{Theorem}[section]
\newtheorem{lemma}[theorem]{Lemma}
\newtheorem{proposition}[theorem]{Proposition}

\theoremstyle{definition}
\newtheorem{remark}[theorem]{Remark}

\begin{document}

\title{Approximation of curves with piecewise constant or piecewise linear functions}

\author{Fr\'ed\'eric de Gournay \and Jonas Kahn \and L\'eo Lebrat}



\begin{abstract}
 In this paper we compute the Hausdorff distance between sets of continuous curves and sets of piecewise constant or linear discretizations. These sets are Sobolev balls given by the continuous or discrete $L^p$-norm of the derivatives. We detail the suitable discretization or smoothing  procedure which are preservative in the sense of these norms. Finally we exhibit the link between Eulerian numbers and the uniformly space knots B-spline used for smoothing.
\end{abstract}

\maketitle

\section*{Introduction}
 This article focuses on a widespread problem of approximation which consists in approaching a curve by a set of points or by a piecewise linear function (line segments or polyline). We also analyze the reverse operation called smoothing, which amounts to, given a set of points or polyline find an approaching curve with an higher level of regularity. These two approaches yields the instinctive question~:~how well can we approximate a particular space of curves with a particular set of points sets or polyline sets.

 This subject has been thoroughly studied, especially by the computer vision community \cite{dunham1986optimum,tomek1974two,mumford1989optimal}. The most common approach is to find minimal length objects controlling some approximation error and the limit error. Our view is different, since we want to approximate in a Hausdorff sense, that is to approximate each curve by a set of points or polylines, and each set of points or polylines by a curve, so that all approximations are close for an appropriate distance.

  In practice, the authors have encountered this question when trying to computationally project a measure on a space of pushforward measures of curves \cite{boyer2016generation,chauffert2017projection,Fred2018troisquart,Fred2018Projections}: the implementation needs a discretization, and it is guaranteed to work only if both directions of approximation are small for the transportation distance $\mathcal{W}^1$. It is likely that this kind of results may be useful in other contexts in computer science.

  In this article, we prove that Sobolev balls and similar spaces may be approximated by discretized Sobolev spaces, where the norm is given by discrete derivatives. As explained in the notations, the Hausdorff distance comes from the transportation distance on both time and space, giving a very robust meaning to the approximation.

  An ingredient in the proof is exhibiting a Sobolev curve that approximates a given set of points. The construction makes use of Eulerian numbers. Given their known connection to B-splines \cite{he2012eulerian,wang2008spline}, this might not be so surprising.
  
\section{Notation}

Throughout the paper, $\kappa_a$ will denote a constant depending only on $a$ that might change from line to line.

Curves are in $\mathbb{R}^d$ and we identify discretization of curves with families of vectors $\p=(\p_0,\dots,\p_{n-1})$. Even if we tackle both periodic and non-periodic cases, the notations are tailored for the periodic case, which allows the abuse of notation $\p_{i} = \p_{ i \ (\text{mod }n)}$ for each $i \in \mathbb{Z}$.
In this setting, the discrete convolution product for a family of vectors $\p \in \mathbb{R}^{n\times d}$ and $K \in \mathbb{R}^n$  reads as :
\[
 (K\star\p)_i = \sum_{j = 0}^{n-1} K_j\p_{i-j},
\]
Given any norm $\Vert \bullet \Vert$ in $\mathbb{R}^d$, the discrete renormalized $\ell^q$ norm is defined as
\[\Vert \p\Vert_{\ell^q}=\left(\frac{1}{n} \sum_{i=0}^{n-1} \left\| \p_i \right\|^q \right)^{1/q}.
\] 
Note that this special choice of renormalization of the $\ell^q$ norm turns  Young's convolution inequality into :
\begin{equation}\label{eq:YoungIneq}
\| K\star\p\|_{\ell^q} \leq n \| K\|_{\ell^1} \| \p\|_{\ell^q} \quad \forall q \ge 1. 
\end{equation}
The convolutional discrete derivative operator, $\Delta$ is defined by
\[(\Delta \star \p)_i=\p_{i}-\p_{i-1}.\]
Similarly for any $m\in \mathbb{N}$, the $m$-order discrete convolutional derivative operator $\Delta^{\star m} \in \mathbb{R}^{n}$ is defined by the recursion formula $\Delta^{ \star (m+1)} = \Delta \star \Delta^{\star m}$ with $\Delta^{ \star 1}=\Delta$. Its closed form is given by :

\begin{equation}\label{eqn:discreteDelta}
 \Delta^{\star m}_{i} = (-1)^i { m \choose i}.
\end{equation}

Let us also define $\mathds{1}$ as the identity for the convolution, and $T$ the shift operator :
\[
 \mathds{1} = \begin{cases}
               \mathds{1}_0 = 1 \\
               \mathds{1}_i = 0 \quad i \neq 0
              \end{cases} \quad \text{and} \quad T = \begin{cases}
               T_1 = 1 \\
               T_i = 0 \quad i \neq 1.
               \end{cases}
\]
The discrete derivative operator can be written as $\Delta = \mathds{1} - T$.

Given $\bal=(\bal_0,\dots,\bal_m)$ with $\bal_i \in \mathbb{R}^{+*}$, we consider the periodic Sobolev multiballs  $W_\sharp^{m,q}(\bal)$ and their discrete counterparts $\mathcal{P}_{\sharp,n}^{m,q}(\bal)$ defined as :
 \begin{eqnarray*}
 W_\sharp^{m,q}(\bal) &=& \left\{ f \in L^1_\sharp\left( [0,1] \rightarrow \mathbb{R}^d\right)  \text{ s.t. }  \| f^{(r)} \|_{L_{\sharp}^q([0,1])} \leq \bal_r \quad \forall r, 0 \leq r \leq m\right\},\\
 \mathcal{P}_{\sharp,n}^{m,q}(\bal) &=&\left\{ \p \in \mathbb{R}^{n\times d} \quad \text{s.t.} \quad n^r \| \Delta^{\star r} \star \p \|_{\ell^q} \leq \bal_r \quad \forall r, 0 \leq r \leq m \right\},
 \end{eqnarray*}
 where $L^1_\sharp$ is the set of periodic functions in $L^1$ and $f^{(r)}$ denotes the derivative of order $r$ of $f$.
 
 In the non-periodic case we define the Sobolev multiballs as :
\begin{eqnarray*}
 W^{m,q}(\bal) &=& \left\{ f \in L^1\left( [0,1] \rightarrow \mathbb{R}^d\right) \text{ s.t. }  \| f^{(r)} \|_{L^q([0,1])} \leq \bal_r \quad \forall r, 0 \leq r \leq m\right\},\nonumber\\
 \mathcal{P}_n^{m,q}(\bal) &=&\left\{ \p \in \mathbb{R}^{n\times d} \ \text{s.t.} \ n^r 
 \left(\sum_{i=r}^{n-1} \frac{1}{n} \|(\Delta^{\star r} \star  \p)_i \|^q\right)^{1/q} \!\!\!\!\!\!\! \leq \bal_r \ \forall r, 0 \leq 
 r \leq m \! \right\}\!.
\end{eqnarray*}

We consider two different discretizations of curves : given a family of points $\p$ the $0$-spline discretization is defined by :

\[
s^0(\p) : t \mapsto \p_{\lfloor nt  \rfloor},
\]
which simply amounts to considering the piecewise constant function with plateaus on the intervals $[ \frac{i}{n},\frac{i+1}{n} ] , 0 \leq i \leq n-1$. On the other hand the $1$-spline discretization is the linear interpolation between the points, it is defined by :
\[
s^1(\p) : t \mapsto \p_{\lfloor nt  \rfloor} + \{nt\} \left(\p_{\lceil nt  \rceil} - \p_{\lfloor nt  \rfloor}\right),
\]
where $\lfloor \bullet  \rfloor$, $\lceil \bullet  \rceil$ are respectively the floor and ceiling function; we denote the decimal part of a number as : $\{nt\}=nt-\lfloor nt  \rfloor \in [0,1]$.
We introduce the Sobolev multiballs of $0$-splines and of $1$-splines as :
\[
 \mathcal{S}^{m,q}_n(\bal) = \left\{ s^0(\p) \quad \text{with} \quad \p \in \mathcal{P}_n^{m,q} (\bal) \right\},
\]

\[
 \mathcal{L}^{m,q}_n(\bal) = \left\{ s^1(\p) \quad \text{with} \quad \p \in \mathcal{P}_n^{m,q} (\bal) \right\},
\]
with of course $ \mathcal{S}^{m,q}_{\sharp,n}(\bal)$ and $\mathcal{L}^{m,q}_{\sharp,n}(\bal)$ their periodic counterparts.

Finally, we specify a metric between curves. Given curves $f$ and $g$ from $[0,1]$ to $\mathbb{R}^d$, the distance between $f$ and $g$ is defined as :
\begin{equation}\label{eqn:distanceTimeSpace}
 d(f,g) = \int_0^1 \| f(t) - g(t) \| dt.
\end{equation}

The distance~\eqref{eqn:distanceTimeSpace} enforces that the set of values of $f$ and $g$ are similar but also that their time parameterizations are close. The distance~\eqref{eqn:distanceTimeSpace} is related to the 1-Wasserstein distance, if one considers :
\begin{equation*}
  \tilde{f}(t) = \left( f(t),t\right) \text{ and } c((x_1,t_1),(x_2,t_1)) = \| x_1-x_2 \| + | t_1 - t_2 |.
\end{equation*}
Denote  $d_\lambda \tilde{f}$ the push-forward of the Lebesgue measure $\lambda$ of $[0,1]$ on $\mathbb{R}^{d+1}$, that is :
\[
 \text{For any Borel set $A\subset \mathbb{R}^{d+1}$,} \ \quad d_\lambda \tilde{f} (A) = \lambda\left(\tilde{f}^{-1}(A)\right).
\]
Note that, for instance 
\[d_\lambda \tilde s^0(A,B)=\sum_{i=0}^{n-1} \frac{1}{n}\delta_{\p_i}(A)\lambda(B\cap [\frac{i}{n},\frac{i+1}{n}]) \quad \forall A \subset \mathbb{R}^p, B\subset \mathbb{R}.\]

Introduce the $1$-Wasserstein distance between the corresponding measures $d_\lambda \tilde{f}$ and $d_\lambda \tilde{g}$. We have
\[
\mathcal W^1(d_\lambda \tilde{f},d_\lambda \tilde{g})=\inf_{\gamma \in \Pi}\int_{\mathbb{R}^{d+1}\times\mathbb{R}^{d+1} } c((x,t_x),(y,t_y)) d\gamma(x\times t_x,y \times t_y),
\]
where $\Pi$ is the set of measures on $\mathbb{R}^{d+1}\times\mathbb{R}^{d+1}$ whose first and second marginals are given by $d_\lambda \tilde{f}$ and $d_\lambda \tilde{g}$ respectively. One such coupling $\gamma$ is given by the time parameter of the curve, so that :
\begin{equation}\label{eq:MajorW1}
\mathcal W^1(d_\lambda \tilde{f},d_\lambda \tilde{g}) \leq \int_{0}^{1} \left\| f(t) - g(t) \right\| dt.
\end{equation}


\section{Main result}
Introduce the Hausdorff distance between to sets of functions $A$ and $B$ by :
\[
 d_{\mathcal{H}}(A,B) = \sup_{f \in A} \inf_{g \in B} d(f,g) + \sup_{g \in B} \inf_{f \in A} d(f,g),
\]
where $d$ is defined in \eqref{eqn:distanceTimeSpace}.
Our main theorems are stated as follows,

\begin{theorem}\label{theo:firstOrderperiodic}
If $m \ge 1$, the Hausdorff distance between the multi-balls of radii $\bal$ of zero-order periodic splines and the multi-balls of radii $\bal$ of periodic Sobolev functions is bounded by $\frac{\kappa_{\bal}}{n}  $.
\end{theorem}

\begin{theorem}\label{theo:firstOrdernonperiodic}
If $m \ge 1$, the Hausdorff distance between the multi-balls of radii $\bal$ of zero-order non-periodic splines and the multi-balls of radii $\bal$ of non-periodic Sobolev functions is bounded by $\frac{\kappa_{\bal}}{n}  $.
\end{theorem}

\begin{theorem}\label{theo::second-order}
If $m \ge 2$, the Hausdorff distance between the multi-balls of radii $\bal$ of first-order periodic splines and the multi-balls of radii $\bal$ of periodic Sobolev functions is bounded by $\frac{\kappa_{\bal}}{n^2}  $.
\end{theorem}

\begin{remark}
More precisely, Theorem~\ref{theo::second-order} states that, if $m\ge 2$ for any $f\in W^{m,q}_\sharp(\bal) $, there exists $\p \in  \mathcal{P}_{\sharp,n}^{m,q}(\bal)$ such that
\begin{equation}
\label{eq::bounds}
 \mathcal W_1(f,s^1(\p)) \le \frac{\mathcal \kappa_{\bal}}{n^2}, 
\end{equation}
and for any $\p \in  \mathcal{P}_{\sharp,n}^{m,q}(\bal)$, there exists $f\in W^{m,q}_\sharp(\bal) $ such that \eqref{eq::bounds} holds.
\end{remark}

We first describe the approximant, for the discrete to continuous case, in the following proposition.
\begin{proposition}\label{prop:prop2}
 Given a sequence of points $\p$, define the function $f_{\p}$ by
\begin{equation}\label{eqn:f}
 \forall t \in [0,1[, \ f_{\p} (t) = \sum_{i=0}^{n-1} g_i\left(nt-i \right),\quad
 g_i(x) = \sum_{k=0}^m (C^{m-k}\star \Delta^{\star k}\star \p)_i  \frac{x^k}{k!}
 \chi_{0\leq x<1}, 
\end{equation}
with $\chi_A$ the indicator function of the set $A$.
Then the two following properties are equivalent :
\begin{itemize}
 \item{}{For each $r$ and $i$ the coefficient $C^{r}_i$ satisfies 
 \begin{equation}\label{eq:CoeffCm}
   \forall r \geq 1, \quad C^r_i = \frac{E^r_{i-1}}{r!} \quad \text{and} \quad C^0 = \mathds{1}
  \end{equation}
 where $E^k_{i}$, $k\geq 1$ is the $i$-th Eulerian number of degree $k$.}
 \item{}{The curve $f_\p$ is a spline of order $m$, $m-1$ time continuously differentiable  whose $m^{th}$ order derivative is given by
\begin{equation}\label{eq:derivePlusProfonde}
f_\p^{(m)}(t) = n^m (\Delta^{\star m}\star \p)_i \quad \text{ for } t\in \left]\frac{i}{n},\frac{i+1}{n}\right[.
\end{equation}}
\end{itemize}
\end{proposition}

In the course of the proof of Proposition~\ref{prop:prop2} we prove the following seemingly new recurrence relationship between the Eulerian numbers.
\begin{proposition}\label{prop:newrecurrenceRelationship}
The Eulerian numbers are solution to each of the two recurrence equations:
\begin{align}
    \label{rec1}
    E^m_i &= \sum_{k=1}^m {m \choose k} \sum_{l=0}^{k-1} (-1)^l  { k-1 \choose l} E^{m-k}_{i-1-l}\\
    \label{rec2}
    E^m_{i-1} &= \sum_{k=0}^m {m \choose k} \sum_{l=0}^{k} (-1)^ {l-1}  { k-1 \choose l-1} E^{m-k}_{i-l}
\end{align}
\end{proposition}

\begin{proof}[Propositions~\ref{prop:prop2} and~\ref{prop:newrecurrenceRelationship}]

 Let $f_{\p}$ be the function defined in Equation \eqref{eqn:f}. It is trivial to see that Equation \eqref{eq:derivePlusProfonde} is true if $C^0$ is the convolution identity kernel. It remains to check the regularity at the connections, indeed the $l$-th derivatives on the right and on the left of the spline $f_{\p}$ have to be equal at each connection, that is, for each $i \in \llbracket 0, n-1  \rrbracket$ and $l \in \llbracket 0, m-1 \rrbracket$
\[
 \lim_{t \rightarrow 0} g^{(l)}_{i+1}(t) = \lim_{t \rightarrow 1} g^{(l)}_i(t).
\]
This gives the following equation : 
\begin{align}\label{eqn:Lemma1rec1}
\left( C^{m-l} \star  \Delta^{\star l} \star  \p \right)_{i+1} =& \sum_{k=l}^{m} \frac{1}{(k-l)!}\left( C^{m-k} \star  \Delta^{\star k} \star  \p \right)_{i} \nonumber \\
=& \sum_{k=0}^{m-l} \frac{1}{k!}\left( C^{m-l-k} \star  \Delta^{\star (k+l)} \star  \p \right)_{i},
\end{align}
Since $\p$ is arbitrary, it can be removed, this yields : 
\begin{equation*}
 \sum_{k=0}^{s} \frac{1}{k!}\left( T\star C^{s-k} \star  \Delta^{\star k}  \right) = C^{s}, 0 \leq s \leq m.
\end{equation*}
Subtracting the first term of the sum from the right hand side one has 
\begin{equation}\label{eqn:Lemma1rec2}
 \sum_{k=1}^{s} \frac{1}{k!}\left( T\star C^{s-k} \star  \Delta^{\star k}  \right)_i = (C^{s} - T\star C^{s})_i = C^{s}_i - C^{s}_{i-1} = (\Delta\star C^{s})_i.
\end{equation}

Finally, $f_{\p}$ is $m-1$ continuously differentiable if and only if the coefficient $C$ verify the recursive formula:
 \begin{equation}\label{eq:compatibilitySpline}
 \forall \  1 \leq s \leq m, \quad  \sum_{k=1}^{s} \frac{1}{k!}\left( T\star C^{(s-k)} \star  \Delta^{\star k-1}  \right) = C^{ s}.
\end{equation}

 We turn our attention to solving  Equation~\eqref{eq:compatibilitySpline} and to obtain that for all $s$,
 \begin{equation*}
  C^s_i = \frac{E^s_i}{s!} \quad \text{with} \quad E^s_i= \sum_{k=0}^{i} (-1)^k {s+1 \choose k} (i -k)^s \quad \text{ and } E^0_i = 1 \text{ iff } i=0,
 \end{equation*}
 where $E^k_{i+1}$, $k\geq1$ is the $i-th$ Eulerian number~\cite{comtet2012advanced} of degree $k$. Suppose the formula for $C^r$ is valid up to $r=s-1$, replacing $C^{s-k}$ by its value in Equation~\eqref{eq:compatibilitySpline} and replacing $\Delta^{\star k}$ by its expression \eqref{eqn:discreteDelta} one obtains:
\begin{equation*}
 s!C^s_i = \sum_{k=1}^s {s \choose k} \sum_{l=0}^{k-1} (-1)^l  { k-1 \choose l} \sum_{r=0}^{i-1-l} (-1)^r {s+1 -k \choose r} (i-1-l-r)^{s-k}.
\end{equation*}
Changing the index $r$ by $q=r+l$ and extending the summation of $q$ from $l$ to $0$ one gets
\begin{equation*}
s!C^s_i =\sum_{q = 0}^{i-1} (-1)^q \sum_{k=1}^s {s \choose k} (i-1-q)^{s-k} \sum_{l=0}^{k-1} { k-1 \choose l } { s+1 -k \choose q-l}.
\end{equation*}
Summing in $l$ the right hand side, one has
\begin{equation*}
s!C^s_i =\sum_{q = 0}^{i-1} (-1)^q { s \choose q} \sum_{k=1}^s { s \choose k } (i-q -1)^{s-k}.
\end{equation*}
Now using the binomial theorem,
\begin{equation*}
s!C^s_i =\sum_{q = 0}^{i-1} (-1)^q { s \choose q} \left((i-q)^s - (i-(q+1))^s\right),
\end{equation*}
an Abel transform gives
\begin{equation*}
s!C^s_i =\sum_{q = 1}^{i-1} (-1)^q (i-q)^s   \left[{s \choose q} + {s\choose q - 1} \right] + 1.
\end{equation*}
Finally, Pascal's rule yields formula~ $s!C^s_i=E^s_i$ and the proof of Proposition~\ref{prop:prop2} is finished. $ \square \quad$

In order to prove Proposition~\ref{prop:newrecurrenceRelationship}, we rewrite Equations~\eqref{eqn:Lemma1rec1},\eqref{eqn:Lemma1rec2} where we substitute $C$ with the corresponding Eulerian number $E$ using formula~\eqref{eq:CoeffCm}.
\end{proof}

\begin{proposition}\label{prop:prop4}
 The periodic spline $f_{\p}$ satisfying Equalities~\eqref{eq:CoeffCm},\eqref{eq:derivePlusProfonde} can be expressed in the B-Spline basis. It turns out that the control points are exactly the $\p_i$ :
\begin{align} \label{eq:BsplinePolynomial}
 f_\p(t) &= \sum_{i\in \mathbb{Z}}  B^{m}(nt-i) \p_i\\
 \text{ with } \quad B^{m}(x) &= \frac{1}{m!}\sum_{k=0}^{m+1} (-1)^k {m + 1 \choose k } (x-k)^m_{+},\nonumber
\end{align}
where $(a)^m_+$ denotes the $m$-th power of the positive part of $a$. Note that the formula defined is periodic and the support of $B^{m}$ is $[0,m+1[$, so that, for a fixed $t$, $f_\p$  is a finite sum.
\end{proposition}

\begin{proof}[Proposition~\ref{prop:prop4}]
 The function which satisfies condition \eqref{eq:derivePlusProfonde} is defined up to the addition of a polynomial of degree $m-1$. This polynomial has to be periodic then it is reduced to a constant. Finally the homogeneity of $f_{\p}$ in $\p$ yields its uniqueness.

It remains to show that the function defined in \eqref{eq:BsplinePolynomial} satisfies condition~\eqref{eq:derivePlusProfonde}.
Using the differentiation formula of the equispaced B-spline~\cite{schumaker2007spline}, the $k$-th derivative of $f_\p(t)$ is given by :
\begin{equation*}
 f^{(k)}_\p(t) = \sum_{i \in \mathbb{Z}} \p_i n^k \sum_{j=0}^{k} (-1)^j {k \choose j} B^{m-k} (nt-i-j)
\end{equation*}
Since $B^0(t) = \chi_{t\in[0,1[}$ $m$-th derivative reads as :
\begin{equation*}
 f^{(m)}_\p(t) = \sum_{i\in \mathbb{Z}} \p_i n^m \sum_{j=0}^{m} (-1)^j {m \choose j} \chi_{t\in[\frac{i+j}{n} ,\frac{i+1+j}{n}[}.
\end{equation*}

\begin{equation*}
 f^{(m)}_\p(t) = \sum_{i \in \mathbb{Z}}  n^m \sum_{a=0}^{m} (-1)^a {m \choose a} \p_{i-a} \chi_{t\in[\frac{a}{n},\frac{a}{n}[}.
\end{equation*}
For $t$ in $[\frac{i}{n},\frac{i+1}{n}]$ one has
\begin{equation*}
 f^{(m)}_\p(t) = n^m \sum_{k=0}^m (-1)^k {m \choose k} \p_{i-k} = n^m \left(\Delta^{\star m } \star \p \right)_i
\end{equation*}
which allows to conclude.
\end{proof}

\section{Proof of the theorems}
This section deals with the proofs of the main theorems. It is subdivided into $4$ sections. In Section~\ref{sec:notationTechnicalLemmas}, we introduce some useful results and operators used throughout the rest of the proof. In Section~\ref{sec:approximationBySplines}, we construct the spline approximation when the continuous curve is given and show that the distance between the spline and the continuous curve is bounded with the correct rate with respect to $n$. In Section~\ref{sec:approximationByFunctions}, we construct a continuous curve when the spline approximation is given. In Section~\ref{sec:wassersteinDistance} the distance between constructed continuous curve and the given spline approximation is proven with the correct rate but the continuous curve does not belong to the correct multi-ball. Finally in Section~\ref{sec:proofOfTheorems}, we gather the results of the different sections and prove the main theorems.  
\subsection{Notations and technical lemmas}\label{sec:notationTechnicalLemmas}

In the following, we make extensive use of the shift operator $\sigma_m$ defined as 
\begin{equation}\label{eq:kernelDelay}
 \begin{cases}
(\sigma_m \star  \p)_i = \p_{i+\frac{m+1}{2}}& \text{ if }m\text{ is odd} \\ 
&\\
(\sigma_{m} \star  \p)_i = \frac{1}{2} \left( \p_{i+ m/2} + \p_{i+ m/2 +1} \right) & \text{ if }
m \text{ is even}\end{cases}.
\end{equation}

Moreover, we need a notion of support of the convolution kernel, this notion is well suited to the non-periodic case and is only useful in this context.
\begin{lemma}\label{lemma:normShifted} Let $\alpha ,\beta\in \mathbb{N}$ with $\alpha+\beta < n$, we say that a kernel $K$ has support in $[-\alpha,\beta]$ if  $K_i =0$ for each $\beta<i<n-\alpha$. For such a kernel $K$, we have, for all $A\in \mathbb{R}^n$, for all $a\ge \beta$ and $b< n-\alpha$,
 \[
  \left(\sum_{i=a}^b \left\| (K\star A)_i \right\|^q\right)^\frac{1}{q} \leq n\Vert K\Vert_{\ell^1}\left(\sum_{i=a-\beta}^{b+\alpha} \left\| A_i \right\|\right)^\frac{1}{q}.
 \]
\end{lemma}

Finally we introduce the operator $\Delta^{-1}$, the inverse of the operator $\Delta$.
\begin{lemma}\label{lemma1ODDO} Let $\alpha ,\beta\in \mathbb{N}$. If $A$ has support in $[-\alpha,\beta]$ and verifies $\sum_{i=0}^{n-1}A_i=0$, define $\Delta^{-1}(A)$ as
\[\Delta^{-1}(A)_i =\sum_{j=0}^i A_j-\sum_{j=0}^{\beta} A_j.\]
Then $\Delta \star \Delta^{-1}(A)=A$ . Moreover $\Delta^{-1}(A)$ has support in 
 $[-\alpha,\beta-1]$ and $\Vert \Delta^{-1}(A)\Vert_{\ell^1}\le (\beta+\alpha) \Vert A\Vert_{\ell^1}$.
\end{lemma}

%
%
%


 Next we gather some results about the Eulerian numbers in the following lemma.
\begin{lemma}\label{lemma:symetric}
 For $m \geq 2$ the kernel $C^m$ sums up to 1 and have support in $[0,m]$ and the kernel $C^m \star \sigma_m$ is symmetric. Moreover if 
   \[A=C^m \star \sigma_m - \mathds{1}, \]
then $\Delta^{-2}(A)$ exists. For $m=0,1$, then $A=0$.
\end{lemma}

\begin{proof}[Lemma~\ref{lemma:symetric}]
First we recall the following well-known properties of Eulerian numbers, valid for  $m \geq 1$, see \cite{comtet2012advanced}
\[  E^m_k = E^m_{m-k+1},
 \quad \text{ and }
   E^m_k = (m-k+1) E^{m-1}_{k-1} + k E^{m-1}_k
\]
We prove $\sum_{k=1}^{m+1} C^m_k=\sum_{k=0}^m \frac{E^m_k}{m!} = 1$ by an induction on $m$.
\begin{eqnarray*}
  \sum_{ 1 \leq k \leq m} E^m_k &=& \sum_{ 1 \leq k \leq m}  (m-k+1) E^{m-1}_{k-1} + k E^{m-1}_k\\
  &=& \sum_{k=1}^{m-1} (m-k) E^{m-1}_{k} + k E^{m-1}_k = m \sum_{k=1}^{m-1} E^{m-1}_{k} = m! \quad  \square
 \end{eqnarray*}
 We now study the symmetry of $C^m \star \sigma_m$. If $m$ is odd this property is a direct consequence $E^m_i = E^m_{m-i+1}$. If $m$ is even, simply notice that 
 \begin{eqnarray*}
  (C^m \star \sigma_m)_i &=& \frac{1}{2} \left(C^m_{i + \frac{m}{2}} + C^m_{i + \frac{m}{2}+1}\right)\\
		       &=& \frac{1}{2} \left(C^m_{-i + \frac{m}{2}+1} + C^m_{-i + \frac{m}{2}}\right)= (C^m \star \sigma_m)_{-i}.
 \end{eqnarray*}The coefficients of $A$ sum up to $0$ and $A$ has support in $[-\left\lfloor\frac{m}{2}  \right\rfloor,\left\lfloor\frac{m}{2}  \right\rfloor]$ so that $\Delta^{-1}(A)$ exists. For $n$ large enough, we now prove that the symmetry of $A$ ensures that the coefficients of  $\Delta^{-1}(A)$ sum up to $0$. For that purpose, for each $j$, denote $u=n-j$ so that $A_u=A_j$.
 \begin{align*}
  \sum_{i=0}^{n-1} \Delta^{-1}(A)_i &= \sum_{i=0}^{n-1} \left(\sum_{j=0}^i A_j \right) - n\sum_{j=0}^\beta A_j =  \sum_{j=0}^{n-1} (n-j) A_j  - n\sum_{j=0}^\beta A_j\\
  &= \frac{1}{2} \left( \sum_{u=1}^{n} u A_{u} + \sum_{j=0}^{n-1} (n-j) A_{j} \right) - \frac{n}{2}\left(\sum_{j=0}^\beta A_j +\sum_{u=n-\beta}^n A_{u} \right)\\
  &= \frac{1}{2} \left( \sum_{u=1}^{n} u A_{u} - \sum_{j=0}^{n-1} j A_{j} \right) - \frac{n}{2}\left(\sum_{j=0}^{n-1} A_j +A_n \right)\\
  &= \frac{n}{2} A_n - \frac{n}{2} A_n = 0.
 \end{align*}
Hence $\Delta^{-2}(A)$ exists and has support in $[-\left\lfloor\frac{m}{2}  \right\rfloor,\left\lfloor\frac{m}{2}  \right\rfloor - 1]$. Moreover,we have
 \[
 \Vert\Delta^{-2}(A)\Vert_{\ell^1}\le \frac{\kappa_m}{n} \quad \text{ for } n \geq m.
 \]
\end{proof}

\subsection{Approximation of function by splines}\label{sec:approximationBySplines}
We now describe the approximating spline for a continuous curve in the periodic and non-periodic case. We prove that the approximations belong to the correct Sobolev multi-balls. We also prove that the $\mathcal{W}^1$-distance between this approximation and the continuous curve is bounded with the correct rates.

\begin{proposition}\label{prop:prop1}
 Let $f  \in W_\sharp^{m,q}(\bal)$, (resp. $f\in W^{m,q}(\bal)$),  define $\p\in \mathbb{R}^{n\times d}$  as $\p_i = f\left(\frac{i}{n} \right)$ for all $i=0\dots n-1$ .
 \begin{itemize}
 \item{}{Then $\p \in \mathcal{P}_{\sharp,n}^{m,q}(\bal)$ (resp. $\mathcal{P}_{n}^{m,q}(\bal)$)}
  \item{}{If $m \ge 1$, then  $\displaystyle{ d(f,s^0(\p)) \leq \frac{\bal_1}{n}}$.}
   \item{If $m \ge 2$, then $\displaystyle{ d(f,s^1(\p)) \leq \frac{\bal_2}{n^2}}$.}
 \end{itemize}
 This proposition states that the distance from the multi-balls of Sobolev functions to the set of splines  behaves exactly as stated in Theorems~\ref{theo:firstOrderperiodic},\ref{theo:firstOrdernonperiodic} and \ref{theo::second-order}.
\end{proposition}

Let $f\in W_\sharp^{m,q}(\bal)$ (resp.  $f\in W^{m,q}(\bal)$) and set  $\p_i=f(\frac{i}{n})$, we first prove that $\p \in \mathcal{P}_{\sharp,n}^{m,q}(\bal)$ (resp. $\p \in \mathcal{P}_{n}^{m,q}(\bal)$).

\begin{proof}
For any $k\le m$, let $i \ge k$ in the case $f\in W^{m,q}(\bal)$ and let $i$ be arbitrary in the case $f\in W_\sharp^{m,q}(\bal)$, we have
 \[
  (\Delta^{\star k}\star \p)_i
  =\frac{1}{n^k} \int_{s_1=i-1}^{i} \int_{s_2 =s_1 -1}^{s_1} \cdots \int_{s_k=s_{k-1} -1}^{s_{k-1}} f^{(k)}\left(\frac{s_k}{n}\right) ds_k \cdots ds_2 ds_1 
  \]
 Notice that $s_k$ is integrated on the interval $[i-k,i]$.   We use a Fubini theorem, and in the periodic case , we use a change of variable $s_k+k\rightarrow s_k$ to obtain :
   
\begin{eqnarray} \Vert n^ k(\Delta^{\star k} \star \p)_i \Vert &\le 
 & \int_{s_k = 0}^{n}   \int_{s_{k-1}=s_{k}}^{s_{k}+1} \dots\int_{s_1=s_2}^{s_2+1} \Vert f^{(k)}\left(\frac{s_k}{n}\right) \Vert \chi_{s_1 \in ]i-1,i]}ds_1\cdots ds_{k-1} ds_{k} \nonumber \\
 &\le& \int_{s_k = 0}^{n}  \left\Vert f^{(k)}\left(\frac{s_k}{n}\right) \right\Vert \theta_i(s_k)ds_{k}. \nonumber
 \end{eqnarray}
In the periodic case, the functions $\theta_i$ are functions that verify
 \[\forall s, 0 \leq \theta_i(s) \leq 1, \ \sum_{i=K}^n \theta_i(s)=1 \quad \text{ and } \int_0^n \theta_i(s)ds \le 1 \quad \forall i \ge K,\]
 where $K=0$ in the periodic case and $K=k$ in the non-periodic case. By Jensen's inequality, we have :
 \begin{eqnarray*}\Vert n^ k(\Delta^{\star k}\star  \p)_i \Vert^q &\le& 
 \left (\int_{s_k = 0}^{n}  \left\Vert f^{(k)}\left(\frac{s_k}{n}\right) \right\Vert \frac{\theta_i(s_k)}{\Vert \theta_i\Vert_{L^1}}ds_{k}\right)^q {\Vert \theta_i\Vert_{L^1}}^q \\
 &\le& \left(\int_{s_k = 0}^{n}  \left \Vert f^{(k)}\left(\frac{s_k}{n}\right) \right\Vert^q \theta_i(s_k) ds_{k}\right) {\Vert \theta_i\Vert_{L^1}}^{q-1}\\
 &\le &\int_{s_k = 0}^{n}  \left \Vert f^{(k)}\left(\frac{s_k}{n}\right) \right \Vert^q \theta_i(s_k) ds_{k}
 \end{eqnarray*}
The $k$-th semi-norm of $\p$ is then bounded by :
 \begin{eqnarray}
 n^{k}\left\lVert \Delta^{\star  k} \star  \p  \right\rVert_{\ell^q} 
 &\le& 
 \left( \sum_{i=K}^{n-1}\frac{1}{n}\int_{s_k = 0}^{n} \left \Vert f^{(k)}\left(\frac{s_k}{n}\right) \right\Vert^q \theta_i(s_k) ds_{k}\right)^{1/q}
 \nonumber\\
 &\leq&  \left( \int_{s = 0}^{n} \left\lVert f^{(k)}\left(\frac s n \right)\right\rVert^q \frac{ds}{n} \right)^{1/q} = \bal_k
 \label{eqn:normDiscrete}
 \end{eqnarray}
This proves that $\p \in \mathcal{P}_{\sharp,n}^{m,q}(\bal)$ in the periodic case, and $\p \in \mathcal{P}_{n}^{m,q}(\bal)$ in the non-periodic case.
\end{proof}
The second item of Proposition~\ref{prop:prop1} involves bounding the distance between $f$ and $s^0(\p)$. In the periodic and non-periodic case, we have
\begin{align}
  d(f,s^0(\p)) &\leq \int_0^1 \left\|f\left(t\right) - f\left(\frac{\lfloor tn\rfloor}{n}\right) \right\|dt \nonumber\\
  &\leq \sum_{i=0}^{n-1} \int_{\frac{i}{n}}^{\frac{i+1}{n}} \left\| \int_{\frac{\lfloor tn\rfloor}{n}}^{t} f'(s)ds  \right\| dt \leq \frac 1 n \sum_{i=0}^{n-1}  \int_{\frac{i}{n}}^{\frac{i+1}{n}} \left\|f'(s)\right\| ds\nonumber\\
  & = \frac 1 n  \left\|f'\right\|_{L^1([0,1])} \le \frac 1 n  \left\|f'\right\|_{L^q([0,1])}  \leq \frac{\bal_1}{n}\nonumber
 \end{align}
The last statement of Proposition~\ref{prop:prop1} amounts to bounding the distance between $f$ and $s^1(\p)$, assuming that $m \ge 2$.

Introducing for each $i$ the point $m_i = \frac{i+1/2}{n}$, and performing a Taylor expansion around this point, we have,  for every $t\in[\frac{-1}{2n},\frac{1}{2n}]$ :
\begin{eqnarray*}
f(m_i+t)&=&f(m_i)+tf'(m_i)+\int_{m_i}^{t+m_i} f''(s) (m_i + t-s) ds \\
s^1(\p)(m_i+t)&=&(\frac{1}{2}-nt)f\left(\frac{i}{n}\right)+(nt+\frac{1}{2})f\left(\frac{i+1}{n}\right) \\
&=&f(m_i)+tf'(m_i)+(\frac{1}{2}-nt)\int_{m_i}^{\frac{i}{n}} f''(s) (m_i- \frac{1}{2n}-s) ds \\
&+&(nt+\frac{1}{2})\int_{m_i}^{\frac{i+1}{n}} f''(s) (m_i+\frac{1}{2n}-s)ds \\
(f-s^1(\p))(t+m_i)&=&
\int_{m_i}^{t+m_i} f''(s)\underbrace{(m_i-s+t)}_{\beta(t,s)}  ds \\
&+&\int_{\frac{i}{n}}^{\frac{i+1}{n}}f''(s)  \left(\underbrace{nt(m_i-s)+\frac{1}{4n}+(\frac{m_i-s}{2}+\frac{t}{2})(\chi_{s\ge m_i}-\chi_{s\le m_i})}_{\gamma(t,s)}\right)ds
\end{eqnarray*}
For $t\in[\frac{-1}{2n},\frac{1}{2n}]$ and the $s$ under consideration, we have
$|\beta(t,s)| \le |t|$ and $|\gamma(t,s)|\le |t|+\frac{1}{2n}$, so that
\[\Vert(f-s^1(\p))(t+m_i)\Vert \le (\frac{1}{2n}+2|t|)\int_{\frac{i}{n}}^{\frac{i+1}{n}}\Vert f''(s)\Vert ds \]
Finally we have
\begin{eqnarray*}
d(f,s^1(\p)) &\leq& \int_0^1 \left\|f(t) - s^1(\p)(t)) \right\|dt\\
&\leq&\sum_{i=0}^{n-1} \int_{t=-\frac{1}{2n}}^{\frac{1}{2n}}(\frac{1}{2n}+2|t|)dt\int_{\frac{i}{n}}^{\frac{i+1}{n}}\Vert f''(s)\Vert ds
=\frac{1}{n^2} \int_{0}^{1} \left\| f''(s) \right\| ds \\
&\leq& \frac{1}{n^2} \left\| f''(s) \right\|_{L^q([0,1])} = \frac{\bal_2}{n^2}.
\end{eqnarray*}
Thus the proof of Proposition~\ref{prop:prop1} is complete. This calculus holds both for periodic and non periodic functions.

\subsection{Approximation of splines by functions}\label{sec:approximationByFunctions}

Now that the spline are known, this section is devoted to the construction of the continuous curve with the correct $\mathcal{W}^1$-distance and the correct Sobolev constants $\bal$.
\subsubsection{Construction of the approximant}
As announced, we have the following proposition
\begin{proposition}\label{prop:prop3}
Let $\p\in\mathcal{P}_{\sharp,n}^{m,q}(\bal)$, 
let $f_{\sigma_m \star \p}$ be as defined in Proposition~\ref{prop:prop2}, then there exists $\kappa_{\bal}$ a constant that depends only on $\bal$ such that 
the spline $f_{\sigma_m \star \p}$ belongs to ${W}_{\sharp}^{m,q}\left((1 + \frac{\kappa_{\bal}}{n^2})\bal\right).$
\end{proposition}

The shift kernel $\sigma_m$ defined in Equation \eqref{eq:kernelDelay} either drifts the indexes of $\p$ or of its mid points $ \frac 12(\p + T*\p)$ depending on the parity of the desired spline. Notice that $\Vert \sigma_m \Vert_{\ell^1}=n^{-1}$ so that for any $\p \in \mathcal{P}_{\sharp,n}^{m,q}(\bal)$, we have $\sigma_m \star \p \in \mathcal{P}_{\sharp,n}^{m,q}(\bal)$ by virtue of Young's convolution inequality~\eqref{eq:YoungIneq}.

The $l-th$ derivative of the spline $f_{\sigma_m \star \p }$ is given by :
\[
 f_{\sigma_m \star \p}^{(l)}(t) = n^l g^{(l)}_i\left( nt - \frac i n \right) \chi_{ t \in [\frac{i}{n},\frac{i+1}{n}]} 
\]
For every $i$ in $\llbracket 0, n -1 \rrbracket$, the $l$-th derivative of $g_i$ reads as :

\[
 \forall t \in [0,1] , g_i^{(l)}\left( t \right) = \sum_{k=l}^m \left( C^{m-k}\star\Delta^{\star k}\star  \sigma_m \star \p \right)_i \frac{t^{k-l}}{(k-l)!}
\]
We first deal with the case $l=m$. In this case
\[\| f^{(m)} \|_{L^q}
= n^{m}\left(\sum_{i=0}^{n-1} \frac{1}{n} \int_0^1 \left\| \left( C^{0}\star \Delta^{\star m}\star \sigma_m \star \p \right)_i \right\|^q dt\right)^{1/q}=\bal_m. \\
 \]
Now suppose that $l \le m-1$. Notice that $T=(0,1,0,\dots)=Id-\Delta$, using Lemma~\ref{lemma1ODDO} the coefficients of $C^{m-l-1}$ and of $C^{m-l}$ sum up to one. Hence the operator $A=\Delta^{-1}(C^{m-l-1} - C^{m-l}\star T^{-1})$ exists and there exists a constant ${\kappa}_{m,l}$ that depends only on $m$ and $l$ such that $\| A \|_{\ell^1} \leq  \frac{{\kappa}_{m,l}}{n}$ and  
\[
C^{m-l-1} = C^{m-l}\star T^{-1} + A\star \Delta, \quad.
\]
Note also that in the case $l=m-1$, one has $A=0$.
Set $\q = \sigma_m \star \p$, by the triangle inequality, we have :
\begin{eqnarray*}
 \| f^{(l)} \|_{L^q} 
 &=& n^{l}\left(\sum_{i=0}^{n-1} \frac{1}{n} \int_0^1 \left\| \sum_{k=l}^m \left( C^{m-k}\star \Delta^{\star k}\star \q \right)_i \frac{t^{k-l}}{(k-l)!} \right\|^q dt\right)^{1/q} \\
  & \leq &
 n^{l}\left(\sum_{i=0}^{n-1} \frac{1}{n} \int_0^1 \left\| \sum_{k=l}^{l+1} \left( C^{m-k}\star \Delta^{\star k}\star \q \right)_i \frac{t^{k-l}}{(k-l)!} - t\left(A\star \Delta^{\star (l+2)} \star \q \right)_i\right\|^q dt\right)^{1/q} : \beta\\
 &+&n^{l}\left(\sum_{i=0}^{n-1} \frac{1}{n} \int_0^1 \left\| \sum_{k=l+2}^m \left( C^{m-k}\star \Delta^{\star k}\star \q \right)_i \frac{t^{k-l}}{(k-l)!} + t\left(A\star \Delta^{\star (l+2)} \star \q\right)_i \right\|^q dt\right)^{1/q} :\gamma
 \end{eqnarray*}
 We claim that the first term, $\beta$, is bounded by $\bal_l$ and that the second term, $\gamma$, scales as $\mathcal O(n^{-2})$. Indeed for the term $\beta$, we have, since $\Delta=Id-T$:
 \[
C^{m-l-1}\star \Delta = -C^{m-l}+C^{m-l}\star T^{-1} + A\star  \Delta^{\star 2}
\]
 so that
 \begin{eqnarray*}
 \beta&= & n^{l}\left(\sum_{i=0}^{n-1} \frac{1}{n} \int_0^1  \left\| (1-t)\left(C^{m-l}\star \Delta^{\star l}\star \q \right)_i + t \left(C^{m-l}\star \Delta^{\star l}\star \q \right)_{i+1} \right\|^q dt\right)^{1/q} \\
&\le & n^{l}\left(\sum_{i=0}^{n-1} \frac{1}{n} \int_0^1  (1-t) \left\| \left(C^{m-l}\star \Delta^{\star l}\star \q \right)_i \right\|^q + t \left\| \left(C^{m-l}\star \Delta^{\star l}\star \q \right)_{i+1} \right\|^q dt\right)^{1/q}. \\
 \end{eqnarray*}
 A change of index in $i$ allows us to conclude 
 \begin{equation}
 \label{eq::finbeta}
 \beta  \le n^{l}\Vert C^{m-l} \star \Delta^{\star l}\star \sigma_m\star \p\Vert_{\ell^q}.
 \end{equation} By virtue of Young's inequality \eqref{eq:YoungIneq} and $\Vert C^{m-l}\star \sigma_m\Vert_{\ell^1}=n^{-1}$, we have  $\beta \le \bal_l$.

 In order to deal with the second term $\gamma$, first assume that $l\le m-2$. Indeed in the case $l=m-1$, we have $A=0$ and then $\gamma=0$ and nothing is to be proven. In the case $l \le m-2$, bound $t$ by $1$, introduce the operator 
 \begin{equation}
 \label{eq::defin:Q}
Q=\sum_{k=l+2}^{m} \frac{1}{(k-l)!}\left | C^{m-k}\star \Delta^{\star(k-(l+2))}\right | + |A|,
\end{equation}
 and note that there exists a constant $\kappa_{m,l}$ that depends only on $m$ and $l$ such that $\Vert Q\Vert_{\ell^1} \le \frac{\kappa_{m,l}}{n}$. Then Young's inequality yields:
 
\[
 |\gamma| \le n^l\Vert Q\star \Delta^{\star(l+2)}\star \q\Vert_{\ell^q}\le \kappa_{m,l}
n^l\Vert \Delta^{\star(l+2)}\star \p\Vert_{\ell^q} \le{ \kappa}_{m,l} \frac{\bal_{l+2}}{n^2}.
\] 
Hence for any $l \le m-1$, we have
\[ \| f^{(l)} \|_{L^q}  \le \bal_l + { \kappa}_{m,l} \frac{\bal_{l+2}}{n^2},\]
and for $l=m$ or $l=m-1$, we have
\[ \| f^{(l)} \|_{L^q}  \le \bal_l.\]

\begin{lemma}
\label{lem:defin:f:non-periodic}
Let $m \ge 1$ and $p \in \mathcal P_{n}^{m,q}(\bal)$.

\noindent For $\theta=1-\frac{20m}{n}$ and $\tau=\frac{10m}{n}$ define :
 \[
  \tilde{f}_\p(t) = f_{\sigma_m \star \p}\left(\theta t + \tau \right),
 \]
where $f_{\sigma_m \star \p}$ is defined in Proposition~\ref{prop:prop2}, then
 $\tilde{f}_\p \in W^{m,q}(\bal+\frac{\kappa_\alpha}{n^2})$
\end{lemma}
\begin{proof}[Lemma~\ref{lem:defin:f:non-periodic}]
 The differentials of $\tilde{f}_{\p}$ are given by :
\[
 \tilde{f}^{(l)}_{\p}(t) = (\theta n)^l g_{i}^{(l)}\left(n\theta t + n\tau - i\right)\chi_{ t \in [i\theta,(i+1)\theta)]},
\]
the Sobolev semi-norm of $\tilde{f}$ can be written as :
\begin{eqnarray*}
 \| \tilde{f}^{(l)}_{\p} \|_{L^q} 
 &=& (n\theta)^{l}\left(\sum_{i = 10 m}^{n-10m} \frac{1}{n\theta} \int_0^1 \left\| \sum_{k=l}^m \left( C^{m-k}\star \Delta^{\star k}\star \sigma_m \star \p\right)_i \frac{t^{k-l}}{(k-l)!} \right\|^q dt\right)^{1/q} \\
 &\leq& \theta^{-\frac{1}{q}}\tilde{\beta} +\theta^{-\frac{1}{q}}\tilde{\gamma}, 
 \end{eqnarray*} 
 where we used $|\theta |< 1$ to obtain the last bound and where the variables $\tilde{\beta}$ and $\tilde{\gamma}$ are similar to  $\beta$ and $\gamma$ defined in proof of Proposition~\ref{prop:prop3}, other than their sums in $i$ range from $10m$ to $n-10m$.
 
 We can then follow the same outline of the proof of Proposition~\ref{prop:prop3}; using Lemma~\ref{lemma:normShifted} we can verify that the $\sigma_m$ shift does not interfere with non-periodicity of $\p$, owing to the sufficiently large buffer $\tau$. Then, one has similar bounds :
 \begin{equation}\label{eq:openLq}
  \| \tilde{f}^{(l)}_{\q} \|_{L^q} \leq \theta^{-\frac{1}{q}}\left(\bal_l + \kappa_{m,l} \frac{{\bal}_{l+2}}{n^2}\right).
 \end{equation}
Now using that $\theta^{-1}{q} \le 1+\kappa /n$, one can conclude that
  \[\| \tilde{f}^{(l)}_{\q} \|_{L^q} \leq \bal_l +  \frac{\kappa_{m,l,\bal}}{n}.\]
\end{proof}

\subsection{Wasserstein distance}\label{sec:wassersteinDistance}

It remains to bound the distance $d$ between the piecewise constant or linear discretization and $f$ the continuous approximant built with the vector $\p$.

\begin{lemma}\label{lemma:5}
Let $m \ge 1$ and $p \in \mathcal P_{\sharp,n}^{m,q}(\bal)$  and let $f_{\sigma_m \star \p}$ be defined as in Proposition~\ref{prop:prop2}, then 
\[ d(f_{\sigma_m \star \p},s^0(\p)) \leq \frac{\kappa_{\bal}}{n}.\]
\end{lemma}

The distance $d$ between $f_{\sigma_m \star \p}$ and $s^0(\p)$ is bounded by :
 \begin{align}
  d(f_{\sigma_m \star \p},s^0(\p)) &=  \sum_{i=0}^{n-1} \int_{\frac{i}{n}}^{\frac{i+1}{n}} \left\| g_i\left(nt-\frac{i}{n}\right) - \p_i \right\| dt \label{eqn:W1forfq}\\
  &= \sum_{i=0}^{n-1} \frac{1}{n} \int_{0}^{1} \| g_i(u) - \p_i \| du.
 \end{align}
Now using the triangle inequality, one has
 \begin{align}
 \left\| g_i(t) - \p_i \right\|  &\le  \left\| \left(C^{ m}\star \sigma_m \star \p \right)_{i}  - \p_{i} \right\| + \sum_{k=1}^m \left\|\left(C^{m-k}\star \Delta^{\star k}\star \sigma_m \star \p \right)_{i}\frac{t^k}{k!} \right\| \label{eq:constantPart}
\end{align}
Integrating in $t$ at the first line and summing in $i$ at the second line, allows us to use Young's inequality for the third line given the fact that $\| C^{m-k} \star \sigma_m \|_{\ell^1} \leq 1/n$. Now for the last line, using that the $\ell^1$-norm is lower than $\ell^p$-norm (by Jensen's inequality),  one can conclude that second term of~\eqref{eq:constantPart} is bounded by :
\begin{align}
 &\displaystyle \sum_{i=0}^{n-1}\frac{1}{n} \int_{0}^{1} \left\| \sum_{k=1}^m \left(C^{m-k}\star \Delta^{\star k}\star \sigma_m \star \p \right)_{i}\frac{t^k}{k!} \right\| dt \nonumber\\ 
 &\displaystyle\leq \sum_{i=0}^{n-1}\frac{1}{n}\sum_{k=1}^m \left\| \left(C^{m-k}\star \Delta^{\star k}\star \sigma_m \star \p \right)_{i}\frac{1}{(k+1)!} \right\|\nonumber\\
 &\displaystyle\leq \sum_{k=1}^m \left\| \left(\Delta^{\star k} \star \p \right)\right\|_{\ell^p}\frac{1}{(k+1)!} \nonumber\\
& \displaystyle \leq \frac{\bal_1}{n} + \frac{\kappa_{\bal}}{n^2}. \label{eq:SumW1Cmmk}
\end{align}

It remains to deal with the first term appearing in inequality~\eqref{eq:constantPart} which can be rewritten as
\[
\| \left(C^{ m}\star \sigma_m \star \p \right)_{i}  - \p_{i} \|
=\| K\star \p_{i} \| \quad \text{ with }K= C^m \star \sigma_m - \mathds{1}.
\]
Notice that $K$ sums up to zero, so that there exists $A=\Delta^{-1}(K)$ with $\Vert A\Vert_{\ell^1}\le \frac{{ \kappa}_{m}}{n}$ for some constant $\kappa_{m}$. As a result, 

\begin{eqnarray}
\sum_{i=0}^{n-1}\frac{1}{n} \int_{0}^{1} \left\| \left(C^{ m}\star \sigma_m \star \p \right)_{i}  - \p_{i} \right\| dt
&\le & \Vert K\star \p\Vert_{\ell^1} \le \Vert K\star \p\Vert_{\ell^q} \nonumber \\
&\le& { \kappa}_{m} \Vert \Delta \p\Vert_{\ell^q} \le \frac{\bal_1}{n} { \kappa}_{m} \label{eq:CmqmIDW1}
\end{eqnarray}
Hence, up to another constant ${ \kappa}_m$,
\[d(f_{\sigma_m \star \p},s^0(\p)) \leq \frac{\bal_1}{n} { \kappa}_{m} +\frac{{ \kappa}_{\bal}}{n^2}\]

\begin{lemma}\label{lemma:6}
Let $m \ge 2$ and $p \in \mathcal P_{\sharp,n}^{m,q}(\bal)$  and let $f_{\sigma_m \star \p}$ be defined as in Proposition~\ref{prop:prop2}, then 
\[d(f_{\sigma_m \star \p},s^1(\p)) \leq \frac{\kappa_{\bal}}{n^2}\]
\end{lemma}

The distance $d(f_{\sigma_m \star \p},s^1(\p))$ is given by :
 \begin{eqnarray*}
d(f_{\sigma_m \star \p},s^1(\p)) &=&\int_0^1 \left\|f(t) - s^1(\p)(t)\right\|_1 dt\\
&=&  \sum_{i=0}^{n-1} \int_{\frac{i}{n}}^{\frac{i+1}{n}} \left\|g_i\left( nt - \frac{i}{n} \right) - s^1(\p)\right\|_1 dt,
 \end{eqnarray*}
 hereafter we divide the right hand side into three parts, $\beta,\gamma,\delta$ :
 \begin{align*}
 \left\| g_i(t) - s^1(\p)(t)\right\|  &\le  \left\| \left(\left(C^{ m}\star \sigma_{m} - \mathds{1}\right)\star \p\right)_i \right\| :\beta \\
  &+  \left\| \left(C^{ m-1}\star \sigma_{m}\star \Delta\star \p\right)_{i} t  - (\p_{i+1}-\p_{i})t \right\| :\gamma\\
  &+ \left\|\sum_{k=2}^m  \left(C^{m-k}\star \Delta^{\star k}\star \sigma_{m}\star \p\right)_{i}\frac{t^k}{k!} \right\|:\delta\\
\end{align*}
The $\beta$ term is treated in a similar fashion to the previous section. Using Lemma~\ref{lemma:symetric}, there exists a constant ${ \kappa}_{m}$ that depends only on $m$ and a kernel $A$ with $\Vert A\Vert_{\ell^1}\le \frac{{ \kappa}_{m}}{n}$ such that $ A =\Delta^{-2}(C^{ m}\star \sigma_{m} - \mathds{1})$.
Hence $\beta \le \Vert (A\star \Delta^2 \star \p)_i\Vert$.
In order to deal with the $\delta$ term, bound $t$ by $1$, introduce the operator \[Q=\sum_{k=2}^m C^{m-k}\star \Delta^{\star (k-2)}\star \sigma_{m},\]
then $\delta \le \Vert (Q\star \Delta^2\star \p)_i\Vert $. It is easy to check that there exists a constant ${ \kappa}_{m}$ that depends only on $m$ such that $\Vert Q\Vert_{\ell^1}\le \frac{{ \kappa}_{m}}{n}$.
It remains to deal with the $\gamma$ term. For that purpose notice that

\begin{eqnarray*}
\left(C^{ m-1}\star \sigma_{m}\star \Delta\star \p\right)_{i}   - (\p_{i+1}-\p_{i}) &=& \left( C^{ m-1}\star \sigma_{m}\star \Delta\star \p - \sigma_{-3}\star \Delta \star \p \right)_i \\
&=& \left( \left(C^{ m-1}\star \sigma_{m} - \sigma_{-3}\right) \star \Delta\star \p\right)_i .
\end{eqnarray*}
The operator $C^{ m-1}\star \sigma_{m} - \sigma_{-3}$ sums up to zero and has support in $[-m,m]$ so that it is a first order derivative kernel in the sense of Lemma~\ref{lemma1ODDO} and there exists a constant ${\kappa}_{m}$ that depends only on $m$ and a kernel $R = \Delta^{-1}\left( C^{ m-1}\star \sigma_{m} - \sigma_{-3}\right) $
with $\Vert R\Vert_{\ell^1}\le \frac{\kappa_{m}}{n}$, so that $\gamma \le \Vert (R\star \Delta^2 \star \p)_i\Vert $.

Collecting all the terms we have,
\[ \left\| g_i(t) - s^1(\p)(t)\right\|
\le 
\sum_{j=1}^3 \Vert (A^j\star \Delta^2 \p)_i \Vert \quad \text{ with} \Vert A^j\Vert_{\ell^1}\le \frac{{\kappa}_{m}}{n},
\]
and finally
\begin{eqnarray*}
 d(f_{\sigma_m \star \p},s^1(\p)) \leq \frac 1 n \sum_{i=0}^{n-1} \int_0^1 \| g_i(t) - s^1(\p)(t) \| dt
 \leq{\kappa}_{m} \frac{{\bal}_2}{n^2}.
\end{eqnarray*}

\begin{lemma}\label{lemma:7}
Let $m \ge 1$ and $p \in \mathcal P_{n}^{m,q}(\bal)$, let $\tilde{f}_{\sigma_m \star \p}$ be defined as in Lemma~\ref{lem:defin:f:non-periodic}, then
\[d(\tilde f_{\sigma_m \star \p},s^0(\p)) \leq \frac{\kappa_{\bal}}{n}\]
\end{lemma}
\begin{lemma}\label{lemma:8}
Let $m \ge 2$ and $p \in \mathcal P_{n}^{m,q}(\bal)$  and let $\tilde{f}_{\sigma_m \star \p}$ be defined as in Lemma~\ref{lem:defin:f:non-periodic}, then 
\[d(\tilde{f}_{\sigma_m \star \p},s^1(\p)) \leq \frac{\kappa_{\bal}}{n}\]
\end{lemma}

\begin{proof}[Lemmas~\ref{lemma:7}~and~\ref{lemma:8}]
 
Since for any $\p$, $ d(s^{0}_\p,s^{1}_\p)\le \frac{\alpha_1}{n}$, it suffices to prove Lemma~\ref{lemma:7}. We have
\begin{align*}
  d(\tilde{f}_{\sigma_m \star \p},s^0_\p) &= \int_{0}^{1} \left\|\tilde{f}_{\sigma_m \star \p}(t) - s^0_\p(t)   \right\| dt \\
  &\leq  \underbrace{ \int_{0}^{1} \left\|\tilde{f}_{\sigma_m \star \p}(t) - s^0_\p\left( \theta t + \tau \right) \right\| dt}_{\alpha}+ \underbrace{\int_{0}^{1}\left\| s^0_\p\left( \theta t + \tau \right) - s^0_\p(t)  \right\| dt}_{\beta} 
\end{align*}

The $\alpha$ term is a subpart of the equation~\eqref{eqn:W1forfq} and can be bounded in a similar fashion to~\eqref{eq:constantPart} :
\begin{align*}
 \alpha &= \int_0^1 \left \| f_{\sigma_m \star \p}(\theta t + \tau) - s_\p^0(\theta  t +  \tau)  \right\| = \theta^{-1} \sum_{i=10m}^{n-10m} \int_{\frac{i}{n}}^{\frac{i+1   }{n}} \|f_{\sigma_m \star \p}(t) - s^0_\p(t) \| dt \\
 &\leq \theta^{-1}   \left( \underbrace{\frac{1}{n} \sum_{i=10m}^{n-10m}\| (C^{m}\star \sigma_m \star \p)_i -\p_i \|}_{\zeta} + \underbrace{\frac{1}{n}\sum_{i=10m}^{n-10m}\sum_{k=1}^m \| (C^{m-k}\star\Delta^{\star k } \star \sigma_m \star \p)_i \frac{1}{(k+1)!} \| }_{\eta} \right)\\
\end{align*}

Given that the support of $C^{m-k}\star \sigma_m$ is included in $[-m,m]$ and using Lemma~\ref{lemma:normShifted} the $\eta$ term can be bounded in the same manner as in Equation~\eqref{eq:SumW1Cmmk} . For the $\zeta$ part define $A=\Delta^{-1}(C^{m}\star \sigma_m - \mathds{1})$; the support of $A$ is included in $[-m,m]$ by virtue of Lemma~\ref{lemma:normShifted} and bounding $\theta^{-1}$ by $1 + \frac{\kappa}{n}$, we have that :
\[
 \alpha \leq \frac{\kappa_{\bal}}{n}
\]

For the $\beta$ part notice that $\theta + 2\tau = 1$, so that :

\[
 | \theta t + \tau -t | \leq \tau
\]

\begin{align*}
 \beta &= \sum_{i=0}^{n-1} \int_{\frac{i}{n}}^{\frac{i+1}{n}} \left\|s^0_\p(\theta t + \tau) - \p_i \right\| dt \leq \sum_{i=0}^{n-1} \int_{\frac{i}{n}}^{\frac{i+1}{n}} \sum_{k=-n\tau+1}^{n\tau} \| \p_{i+k} - \p_{i + k - 1} \|\chi_{i+k \in \llbracket 1,n-1 \rrbracket} dt \\
 &\leq \frac{1}{n} \sum_{i=1}^{n-1} 2 n \tau \|(\Delta \star \p)_i \| = 2 n \tau \|\Delta \star \p \|_{\ell^1} \leq \kappa \frac{\bal_{1}}{n}
\end{align*}
Since $ \|\Delta \star \p \|_{\ell^1} \leq \frac{\bal_1}{n}$ and  $\tau \leq \frac{\kappa}{n}$.
This allows us to conclude that :
\[
d\left(\tilde{f}_{\sigma_m \star \p},s^0_{\p}\right) \leq \frac{\bal_1}{n} { \kappa}_{m} +\frac{{ \kappa}_{\bal}}{n^2},
\]
and
\[
d\left(\tilde{f}_{\sigma_m \star \p},s^1_{\p}\right) \leq \frac{\bal_1}{n} { \kappa}_{m} +\frac{{ \kappa}_{\bal}}{n^2}.
\]
\end{proof}

\subsection{Proof of theorems}\label{sec:proofOfTheorems}

The end of the proof proceeds as follows. For any $\p \in \mathcal{P}_n^{m,q}(\bal)$, build $\q=\sigma_m \star \p$, notice that $\Vert \sigma_m \Vert_{\ell^1}=n^{-1}$ so that $\q \in \mathcal{P}_{\sharp,n}^{m,q}(\bal)$. 
We have that $f_\q \in W_\sharp^{m,q}(\bal+\frac {\kappa_{\bal}}{n^2})$. Let $\delta \in \mathbb{R}$ be a scaling factor such that $\delta f_\q \in W_\sharp^{m,q}(\bal)$.
Notice that there exists yet another constant depending on ${\bal}$ only and still denoted ${\kappa_{\bal}}$ such that $\delta = 1 + \frac {\kappa_{\bal}}{n^2}$, we then have
\[d(\delta f_\q,f_\q)\le \frac {\kappa_{\bal}}{n^2}\]
By the triangle inequality for the distance $d$, we conclude that
\begin{equation}\label{eqn:correctRate}
d(s^0(\p),\delta f_\q)\le \frac{\kappa}{n}, \quad \text{and } d(s^1(\p),\delta f_\q)\le\frac{\kappa}{n^2} \text{ if } m\ge 2. 
\end{equation}
Thence $\delta f_\q \in W_\sharp^{m,q}(\bal)$ is sufficiently close to $s^0(\p)$ ( resp.  $s^1(\p)$). This ends the proof.

\begin{proof}[Theorems~\ref{theo:firstOrderperiodic}~to~\ref{theo::second-order}]
For any function $f$ in $W_\sharp^{m,q}(\bal)$ (resp. $f$ in $W^{m,q}(\bal)$)  take $\p \in \mathbb{R}^{d\times n}$ such that $\p_i = f\left(\frac{i}{n} \right)$, then $\p \in \mathcal{P}_{\sharp,n}^{m,q}(\bal)$ (resp. $\p \in \mathcal{P}_{n}^{m,q}(\bal)$) by virtue of Proposition~\ref{prop:prop1}. Still using the result of Proposition~\ref{prop:prop1}, the distance between $f$ and its approximant, whether it is a piecewise constant or a piecewise linear spline, is bounded with the correct rate.

Now for any piecewise constant or linear function $s^0(\p) \in \mathcal{S}^{m,q}_{\sharp,n}(\bal)$ or $s^1(\p) \in \mathcal{L}^{m,q}_{\sharp,n}(\bal)$ (resp. $s^0(\p) \in \mathcal{S}^{m,q}_{n}(\bal)$ or $s^1(\p) \in \mathcal{L}^{m,q}_{n}(\bal)$) build $\q = \sigma_m \star \p$ and the smoothing spline $f_\q$ defined as in Proposition~\ref{prop:prop2} (resp.$\tilde f_\q$ defined as in Lemma~\ref{lem:defin:f:non-periodic}). This spline belongs to ${W}_{\sharp}^{m,q}((1 + \frac{\kappa_{\bal}}{n^2})\bal)$ by using Proposition~\ref{prop:prop3} (resp. ${W}^{m,q}((1 + \frac{\kappa_{\bal}}{n^2})\bal)$ by using Lemma~\ref{lem:defin:f:non-periodic}). The distance $d$ between $f_{\q}$ (resp. $\tilde f_\q$) and the piecewise constant or linear spline is bounded and the result of Lemmas~\ref{lemma:5}~or~\ref{lemma:6} (resp. Lemmas~\ref{lemma:7}~or~\ref{lemma:8}) with the correct rates. Introduce the scaled function $\delta f_\q$ (resp. $\delta \tilde f_\q$) as described in~\eqref{eqn:correctRate} to obtain a function in ${W}_{\sharp}^{m,q}(\bal)$ (resp. ${W}^{m,q}(\bal)$) whose distance with respect to the spline is bounded with the correct rate.
\end{proof}

\section*{Conclusion}
In this article, we bound the Hausdorff distance between set of continuous curve with a prescribed Sobelev semi-norm on their derivative and their discrete piecewise constant and piecewise linear counterparts. Bounding the Hausdorff requires a twofold control that is :
\begin{itemize}
 \item given a continuous curve, discretize the curve with a piecewise constant or linear spline sufficiently close in the sense of the 1-Wasserstein distance and which belongs to the suitable spline set.
 \item given a piecewise constant or linear spline, construct a continuous function sufficiently close in the sense of the 1-Wasserstein distance and which belong to the correct Sobolev multiballs.
\end{itemize}
The discretization step is trivial and given by the uniform sampling of the continuous curve. On the over hand finding a $m$ times continuous function that approximates the 0-th or 1-st order spline is trickier. The construction of this continuous approximant involves using $B$-splines of order $m$ but it appears that its expression is elegant (see Proposition~\ref{prop:prop4}). The derivatives continuousness of this approximant yields recurrence relationships involving Eulerian numbers and that are, to the best of our knowledge new.

\begin{appendix}\label{sec:Appendix}


\section{Numerical implementation}

Authors released an open source implementation of the presented smoothing Eulerian B-splines\footnote{\url{https://github.com/lebrat/eulerianApproximation}}. The code implements for $d=2$ the previously presented method with a graphical user interface.  Note that this code is easily scalable to higher dimensions since its it time complexity depends only on the number of points $\p$.

\end{appendix}

\bibliographystyle{plain}
\bibliography{bibi}

\end{document}